\pdfoutput=1
\documentclass[11pt,a4paper,twoside]{article}
\usepackage{amsfonts,amsmath,amssymb}
\usepackage{theorem}
\usepackage{cite}
\usepackage{graphicx}
\usepackage{indentfirst}
\usepackage{placeins}

\usepackage[T1]{fontenc}
\usepackage[utf8]{inputenc}
\usepackage[english]{babel}
\newcommand\PaperShortTitle{Diagonal recovery from noisy Fourier--Hermite coefficients}
\newif\ifsubmission
\submissiontrue
\providecommand\PaperRevision{?}
\providecommand\PaperRevisionDate{\number\year-\ifnum\month<10 0\fi\number\month-\ifnum\day<10 0\fi\number\day}
\providecommand\PaperPreparedBy{(set by make paper-pdf)}
\providecommand\PaperRef{(set by make paper-pdf)}
\InputIfFileExists{revision.tex}{}{}
\usepackage{xcolor}
\definecolor{papermetafill}{HTML}{F4F6F8}
\definecolor{papermetaframe}{HTML}{C7D0D9}
\newcommand\PaperMeta{\ifsubmission\else
  \begin{center}
  \setlength{\fboxsep}{5pt}\setlength{\fboxrule}{0.5pt}%
  \fcolorbox{papermetaframe}{papermetafill}{\footnotesize\foreignlanguage{english}{%
    \begin{tabular}{@{}l@{\hspace{10pt}}l@{}}
    \textbf{Ref} & \PaperRef\\
    \textbf{Date} & \PaperRevisionDate\\
    \textbf{Type} & working draft, edition \PaperRevision; not for distribution\\
    \textbf{Prepared} & \PaperPreparedBy
    \end{tabular}}}%
  \end{center}\fi}
\ifsubmission\else
\usepackage{fancyhdr}
\usepackage{lastpage}
\fancypagestyle{plain}{%
  \fancyhf{}%
  \fancyfoot[C]{\thepage\\[3pt]{\fontsize{6}{7}\selectfont\color{gray}%
    \PaperShortTitle\ \textbullet\ working draft, edition \PaperRevision,
    \PaperRevisionDate\ \textbullet\ \PaperPreparedBy\ \textbullet\ p.~\thepage\ of \pageref{LastPage}\\%
    \PaperRef\ \textbullet\ not for distribution}}}
\AtBeginDocument{\ifdefined\hypersetup\hypersetup{pdfauthor={\PaperPreparedBy}}%
  \else\pdfinfo{/Author (\PaperPreparedBy)}\fi}
\fi

\theorembodyfont{\sl}
\newtheorem{theorem}{Theorem}[section]
\newtheorem{lemma}[theorem]{Lemma}

\newtheorem{remark}[theorem]{Remark}

\numberwithin{equation}{section}

\usepackage{calc}
\numberwithin{theorem}{section}

\newcommand{\card}{\operatorname{card}}
\newcommand{\bxi}{\mbox{\boldmath$\xi$\unboldmath}}

\makeatletter
\renewenvironment{thebibliography}[1]
 {\section*{\centerline{\rm\bfseries References}}%
 \@mkboth{\MakeUppercase\refname}{\MakeUppercase\refname}%
 \list{\@biblabel{\@arabic\c@enumiv}}%
 {\settowidth\labelwidth{\@biblabel{#1}}%
 \leftmargin\labelwidth
 \advance\leftmargin\labelsep
 \@openbib@code
  \usecounter{enumiv}%
  \let\p@enumiv\@empty
  \renewcommand\theenumiv{\@arabic\c@enumiv}}%
  \sloppy
  \clubpenalty4000
  \@clubpenalty \clubpenalty
  \widowpenalty4000%
  \sfcode`\.\@m
  \setlength{\itemsep}{-0.1cm}}
  {\def\@noitemerr
  {\@latex@warning{Empty 'thebibliography' environment}}%
 \endlist}
\renewcommand{\@biblabel}[1]{#1.}
\makeatother \hfuzz=0.5pt \tolerance=500

\begin{document}
\begin{center}
\textbf{Diagonal recovery from noisy Fourier--Hermite coefficients:\\
exact orders of the error}
\end{center}
\ifsubmission\else
\begin{center}
\small Working draft, edition \PaperRevision, \PaperRevisionDate. Not for distribution.
\end{center}
\fi
\PaperMeta

\vspace*{10mm} \centerline{\textsc {S.\,G.\,Solodky $\!\!{}^{\dag}$, Y.\,A.\,Volynets $\!\!{}^{\ddag}$}}

\vspace*{5mm}
\centerline{$\!\!{}^{\dag}\!\!$ Institute of Mathematics, National Academy of Sciences of Ukraine, Kyiv}
\centerline{$\!\!{}^{\ddag}\!\!$ National University of Kyiv-Mohyla Academy, Kyiv, Ukraine}

\vspace*{8mm}
\begin{abstract}
\begin{center}
\parbox{\dimexpr\textwidth-24mm}{\small
We consider the recovery of functions of the weighted Wiener classes $W_s^\mu$ from
perturbed Fourier--Hermite coefficients in the metric of ${\cal L}_\infty(\mathbb R)$.
For the summation methods of the class $\mathcal S^\theta$ we find error estimates
on $W_s^\mu$.
The estimates take into account the behaviour of the Hermite functions both in the oscillatory
region and near the turning point: the maxima of functions with different indices are attained
at different points, and the supremum of a sum over a block of indices is substantially smaller
than the sum of the suprema. Owing to this, in particular,
we establish a common lower bound for all finite
diagonal methods; for $s\ne4/3$ and $p\ne4/3$ it matches the upper bound and
gives the exact order of the error in this class, while at the critical points the gap between
the bounds does not exceed $(\log(\mathrm e/\delta))^{1/4}$ up to a constant.
If, in addition, $s>4/3$ or $p<4/3$,
then this order is also optimal among all recovery operators that use
the whole sequence of data.
}
\end{center}
\end{abstract}

\vspace*{3mm}
\vspace*{3mm}


\vspace*{6mm}

\section{Introduction}

Let a function be given by its Fourier--Hermite coefficients whose values are known with an
error in the $\ell_p$ norm. In this situation it is natural to use algorithms that use a
finite set of such values and multiply each measured coefficient by its own factor.
Such algorithms are called finite diagonal methods, that is, methods that multiply the used
coefficients by arbitrary real factors (cf.~the diagonal linear shrinkers in
\cite[\S\,4.1]{DonohoJohnstone1994}, where the factors are taken from $[0,1]$, and the
coordinatewise linear estimators in \cite[\S\,3]{DonohoLiuMacGibbon1990}). Unlike these works, where
the noise is random and the error is averaged over the noise, here the noise is deterministic: only a
bound on its $\ell_p$ norm is known, and the error is estimated in the worst case both over the
function in the class and over the noise. In this
paper we answer the following question about this class of algorithms: what is the exact order (on the
power scale) of the best guaranteed error if arbitrary finite sets of indices and arbitrary
real factors are allowed?

Optimal recovery of functions and their derivatives from Fourier coefficients given with an
error was studied earlier in \cite{Osipenko1996,MagarilOsipenko2002}. For the trigonometric and
Chebyshev systems, estimates of the minimax error and of the information complexity of numerical
differentiation and summation under deterministic perturbation of the coefficients, order-sharp
on the power scale, were obtained in \cite{SolodkyStasyuk2023,SemenovaSolodky2026}; for
the Laguerre system, in the summation problem in \cite{SolodkyVolynets2026} and in numerical
differentiation in \cite{SolodkyVolynetsDiff}; for
the Legendre system, error estimates of the truncation method in numerical differentiation
were obtained in \cite{Kyselov}.
Recovery of a function from a finite set of Fourier--Hermite coefficients with random
Gaussian noise was considered in \cite{AmatoDeFeis}: there Tikhonov regularization is used,
with the parameter chosen by the GCV criterion, the function is taken from a weighted Sobolev
space with the weight $e^{-x^2}$, and what is estimated is the expectation of the squared deviation
of the regularized solution from a partial sum of the Fourier--Hermite series in the space $L_2$
with the same weight.
Exact errors and optimal methods were found for the recovery of the identity
\cite{BabenkoParfinovychSkorokhodov2021} and diagonal
\cite{BabenkoParfinovychSkorokhodov2024} operators on weighted classes of sequences,
including recovery from the first coordinates known with an error in $\ell_p$; as a
rule, truncation methods with factors turn out to be optimal. The class and the error
are given there in the spaces $\ell_q$, and the results for classes of functions (for example, for fractional
derivatives of periodic functions from their Fourier coefficients) are obtained via the isometry
of the space $L_2$ with $\ell_2$. In our problem the error
is measured in the uniform norm of a function on the whole line, and its order is determined by the
behaviour of the sums of $|h_k(x)|^q$ over $k$, and not by a single weight of the sequence.

Here a feature of the Hermite system shows itself. The maximum of the function $h_k$ is attained near the
point $x_k\approx\sqrt{2k+1}$, so that functions with different indices have their peaks at different
points, and the supremum of a sum over a block of indices is substantially smaller than the sum of the suprema. For example,
for $N$ functions with indices from the interval $[N,3N]$ the sum of their maxima has order
$N^{11/12}$, while the largest, over the signs $\varepsilon_k=\pm1$, value of the norm $\|\sum\varepsilon_k h_k\|_{\infty}$
has order $N^{3/4}$ (Remark~\ref{rem:technique}). The standard estimation scheme, in
which each term is first bounded by the maximum of its function and only then
summed, loses this difference and in some cases gives an exponent worse than the exact
$e_{\mathrm{tr}}$ (the conditions are given in the same remark); for the Laguerre functions the same scheme gave the exact order
\cite{SolodkyVolynets2026}. We apply H\"older's inequality at a fixed $x$ to the
whole block of indices and take the supremum over $x$ after the summation; the resulting
sums are estimated by Theorem~\ref{sum:kernel} (order-sharp) and Lemma~\ref{lem:blockest}.

The proof consists of three parts. First, from the known uniform estimate of the Hermite
functions \cite{Muckenhoupt1970} we derive the upper bound for the power kernel and the block estimate;
the lower bound for the kernel follows from the values $h_{2r}(0)$. Next we
estimate all three terms of the error of a method of $\mathcal S^\theta$ and choose the truncation level
at which these estimates are balanced. Finally, a two-point estimate with a sum of Hermite functions
over indices from the interval $[N,3N]$ (for $s\ge4/3$ and $p>4/3$) or with a single function $h_0$ (for $p\le4/3$) gives
a lower bound for any recovery operator. For $s<4/3<p$ the lower bound is proved
for any finite diagonal method with any set of used indices. To this end
we consider the factors $\nu_k$ with indices from the interval $[M,3M]$, whose length depends
on $\delta$; unused indices correspond to $\nu_k=0$. If at least $M$ of these factors
satisfy $|1-\nu_k|\ge1/2$, then the method noticeably distorts the coefficients of a suitable function of the
class already for exact data. Otherwise at least $M$ factors satisfy $|\nu_k|>1/2$, and
the data error at these indices passes into the result.

The paper is organized as follows. Section 2 contains auxiliary facts and notions.
The main results are stated in Section~\ref{main-results}.
Theorem~\ref{thm:diag} determines the exact order in the whole class of finite
diagonal methods and establishes that for $s>4/3$ or $p<4/3$ this order is
also optimal among all recovery operators, even those using the whole
sequence of data. The exact orders refer to the noncritical values
$s,p\ne4/3$; the lower bounds also hold at the critical points, where the gap does not exceed
a logarithmic factor up to a constant.
For the region $s<4/3<p$ the class of finite diagonal methods is studied here.

\smallskip
\noindent\textit{Terminology.} Below we use several descriptive but not
generally accepted names. The quantity \eqref{sum:law}, that is, the supremum over $x$ of the sum of $|h_k(x)|^q$
over indices $k\le N$, is called the \emph{power kernel}, and the estimate \eqref{eq:blockkernel} of the
same sum over a dyadic block of indices $M\le k<2M$ is called the \emph{block estimate}. By a
\emph{critical point} we mean the value $p=4/3$ (or $s=4/3$) at which a logarithmic factor
appears in the estimate. The standard lower bound for the error will be called the
\emph{two-point estimate}: if the functions $g$ and $-g$ of the class, with admissible errors,
correspond to the same data, then the error of any recovery operator is at least
$\|g\|_\infty$ (Section~\ref{diagonal}).

\section{Preliminaries}\label{prelim}

\subsection{Hermite polynomials and functions}

Denote by ${\cal L}_2(\mathbb R)$ the space of Lebesgue measurable functions $f$ on
$\mathbb R$ with finite norm $\|f\|_2=\bigl(\int_{\mathbb R}|f(x)|^2\,dx\bigr)^{1/2}$, and by
${\cal L}_\infty(\mathbb R)$ the space of essentially bounded measurable functions with the
norm $\|f\|_\infty=\inf\{M\ge0:\ |f(x)|\le M\ \text{almost everywhere on}\ \mathbb R\}$; functions that coincide almost
everywhere are identified.

By $H_n$ and $h_n$ we mean the Hermite polynomials and functions:
\begin{equation}\label{herm}
H_n(x)=(-1)^n e^{x^2}\frac{d^n}{dx^n}e^{-x^2},
\qquad
h_n(x)=\frac{H_n(x)e^{-x^2/2}}{(2^n n!\sqrt\pi)^{1/2}},
\qquad n\in\mathbb N_0 .
\end{equation}
The system $\{h_n \}_{n\ge0}$ is orthonormal and complete in ${\cal L}_2(\mathbb R)$.
The values at zero for even indices are
\begin{equation}\label{zero}
h_{2r}(0)=(-1)^r \pi^{-1/4}
\left(\binom{2r}{r}4^{-r}\right)^{1/2},
\qquad h_{2r+1}(0)=0 ,
\end{equation}
whence, by the Wallis formula,
\begin{equation}\label{wallis}
\pi^{-1/2}\left(r+\tfrac12\right)^{-1/4}
\le |h_{2r}(0)|
\le \pi^{-1/2}r^{-1/4},
\qquad r\ge1 .
\end{equation}
Formula \eqref{herm}, the orthonormality and completeness of the system, and the values \eqref{zero} can be found in
\cite[\S\S\,5.5, 5.7]{Szego}. The estimates of the Hermite functions used below are given
with references in Lemmas \ref{sum:point} and \ref{lem:block}.

For $f\in {\cal L}_2(\mathbb R)$ we set
$$
a_k(f)=\int_{\mathbb R} f(x)h_k(x)\,dx ,
\qquad
\mathbf f=\left(a_k(f)\right)_{k\in\mathbb N_0} ,
$$
and represent the function by its Fourier--Hermite series
\begin{equation}\label{function}
f(x)=\sum_{k=0}^\infty a_k(f)h_k(x) .
\end{equation}
The approximation is sought in the metric of the space ${\cal L}_\infty(\mathbb R)$; for a continuous function
$g$ its norm is
\begin{equation}\label{metric}
\|g\|_\infty=\sup_{x\in\mathbb R}|g(x)| .
\end{equation}
For the expansion \eqref{function} this metric is the weighted uniform metric with the weight natural for the
Hermite polynomials: if $\widetilde H_k$ are the orthonormal Hermite polynomials, then
$\|g\|_\infty=\sup_{x}\bigl|e^{-x^2/2}\sum_k a_k(g)\widetilde H_k(x)\bigr|$. Conditions under
which the series \eqref{function} converges absolutely and uniformly and its sum is continuous are
established in Section \ref{accuracy}; everywhere below where the norm \eqref{metric} is written,
these conditions are satisfied, and equality \eqref{function} is understood pointwise.

\subsection{Function classes and exponents}

Let $1\le s\le\infty$ and $\mu>0$. We introduce the function space
\begin{equation}\label{norm}
W_s^\mu=\Bigl\{f\in {\cal L}_2(\mathbb R):\
\|f\|_{s,\mu}=\Bigl(\sum_{k=0}^\infty \underline{k}^{\,s\mu}|a_k(f)|^s \Bigr)^{1/s}<\infty
\Bigr\},
\qquad 1\le s<\infty ,
\end{equation}
$$
W_\infty^\mu=\Bigl\{f\in {\cal L}_2(\mathbb R):\
\|f\|_{\infty,\mu}=\sup_{k\in\mathbb N_0}\underline{k}^{\,\mu}|a_k(f)|<\infty\Bigr\},
\qquad \underline{k}=\max\{1,k\} .
$$
As in \cite{SolodkyVolynets2026}, the space and its unit ball are denoted by the same
symbol: $W_s^\mu=\{f\in {\cal L}_2(\mathbb R):\|f\|_{s,\mu}\le1\}$ is called a function
class, and what exactly is meant is clear from the context. These classes are built in the same way as the
weighted Wiener classes of periodic functions defined by the Fourier coefficients of the
trigonometric system (see, for example, \cite{Kolom2023}); here the Fourier--Hermite coefficients
are used instead.

The exponent conjugate to $s$ is written as $s/(s-1)$, and similarly for $p$. Everywhere below
\begin{equation}\label{exps}
a(s)=\frac34-\frac1s,
\quad
b(s)=\frac14-\frac1{3s},
\quad
\bar a(s)=\max\{a(s),b(s)\},
\quad
a_+(p)=\max\Bigl\{\frac34-\frac1p,\,0\Bigr\} .
\end{equation}
Since $b(s)=a(s)/3$, we have $\bar a(s)=a(s)>0$ for $s>4/3$, $\bar a(s)=b(s)<0$ for
$s<4/3$ and $\bar a(4/3)=0$; $a_+(p)=a(p)$ for $p>4/3$ and $a_+(p)=0$ for $p\le4/3$.
The logarithmic factor $L_s(N)$ equals $(\log(2+N))^{1/4}$ for $s=4/3$ and one for
$s\ne4/3$; likewise $L_p(N)$. The symbols $c$ and $C$ denote positive constants depending only
on the parameters named in the statement containing them, and never on $N$,
$\delta$, $x$, $f$, the error sequence or the set of indices; their value
may change from line to line.

\subsection{Inexact data and the class of methods $\mathcal S^\theta$}

It is assumed that instead of the exact coefficients $\mathbf f$ the data
$$
\mathbf f^\delta=\mathbf f+\bxi,
\qquad
\mathbf f^\delta=\left(f_k^\delta\right)_{k\in\mathbb N_0} ,
$$
are available, where the error sequence $\bxi=(\xi_k)_{k\in\mathbb N_0}$ consists of
real numbers and is not assumed to be the coefficient sequence of any
function. The error is measured in the $\ell_p$ norm, where
$\|\bxi\|_{\ell_p}=\bigl(\sum_{k\ge0}|\xi_k|^p\bigr)^{1/p}$ for $p<\infty$ and
$\|\bxi\|_{\ell_\infty}=\sup_k|\xi_k|$: for some $1\le p\le\infty$
\begin{equation}\label{perturbation}
\|\bxi\|_{\ell_p}\le\delta,
\qquad 0<\delta<1 .
\end{equation}
All functions, coefficients, errors and factors below are real.

We consider the summation methods
\begin{equation}\label{ModVer}
S_N^\nu\mathbf f^\delta(x)=\sum_{k=0}^N \nu_k^N f_k^\delta h_k(x)
\end{equation}
with triangular arrays $\nu=\{\nu_k^N \}$, $k=0,\ldots,N$, $N\in\mathbb N$. Such methods
are called $\theta$-methods, see, for example, \cite{MathePereverzev2002}. The quality of method
\eqref{ModVer} depends on the truncation level $N$ and on the properties of the array $\nu$.
We assume that for some $C(\nu)>0$ and $\theta>0$
\begin{equation}\label{qual}
|1-\nu_k^N|\le C(\nu)\left(\frac kN\right)^\theta,
\qquad 0^\theta=0,
\qquad k=0,1,\ldots,N .
\end{equation}
Denote by $\mathcal S^\theta$ the class of all methods \eqref{ModVer} satisfying
\eqref{qual}; as a consequence, the array $\nu$ is uniformly bounded. A method of the class
$\mathcal S^\theta$ is everywhere below assumed to be fixed in advance, and the dependence of constants
on the method means the dependence on $\theta$ and $C(\nu)$. Uniformity over the whole class
$\mathcal S^\theta$, on which $C(\nu)$ is unbounded, is not claimed. Condition \eqref{qual}
depends only on the factors, so the examples of such methods (Fej\'er, Abel--Poisson,
Gauss--Weierstrass, Zygmund, Fourier, de la Vall\'ee Poussin) are the same as in
\cite[Example~1.1]{SolodkyVolynets2026}.

A summation method is applied to a function through its coefficient sequence: for
$f\in {\cal L}_2(\mathbb R)$ we write $S_N^\nu f=S_N^\nu(\mathbf f)$, and
\begin{equation}\label{m_F}
S_N f(x)=\sum_{k=0}^N a_k(f)h_k(x)
\end{equation}
is the Fourier summation method, that is, \eqref{ModVer} with $\nu_k^N=1$.

\subsection{Exponents of the diagonal problem}

Everywhere below, along with \eqref{exps}, we use the quantities
\begin{equation}\label{beta}
\beta=\mu-\bar a(s),
\qquad
\rho=\beta+a_+(p),
\qquad
e_{\mathrm{tr}}=\frac{\beta}{\rho} .
\end{equation}
For $\mu>\max\{0,\bar a(s)\}$ both quantities $\beta$ and $\rho$ are positive,
and $e_{\mathrm{tr}}\in(0,1]$; for $p<4/3$ we have $a_+(p)=0$, $\rho=\beta$ and $e_{\mathrm{tr}}=1$.
This exponent describes
the accuracy of finite diagonal methods and, in the regions indicated below, the general minimax accuracy. At the critical values
$s=4/3$ or $p=4/3$ the upper bounds below contain the logarithmic factors $L_s$ and
$L_p$ defined after \eqref{exps}. The base of natural logarithms is denoted by $\mathrm e$, so that
$\log(\mathrm e/\delta)=1+\log(1/\delta)\ge1$ for $0<\delta<1$.

\section{Main results}\label{main-results}

For a finite set $\Omega\subset\mathbb N_0$ denote by $P_\Omega$ the coordinate
projection of a sequence, and by $G_\Omega f=P_\Omega\mathbf f$ the available exact
coefficients. For real factors $\nu=(\nu_k)_{k\in\Omega}$ the diagonal method
\begin{equation}\label{eq:diagmethod}
\psi_\nu(\mathbf y)=\sum_{k\in\Omega}\nu_k y_k h_k
\end{equation}
is defined on every sequence and takes values in $C(\mathbb R)$ as a finite
sum of Hermite functions. Its error is
\begin{equation}\label{eq:localerror}
e_\delta(W_s^\mu,\psi_\nu G_\Omega,C(\mathbb R),\ell_p)
=\sup_{f\in W_s^\mu}\ \sup_{\substack{\bxi\in\ell_p\\\|\bxi\|_{\ell_p}\le\delta}}
\bigl\|f-\psi_\nu(G_\Omega f+P_\Omega\bxi)\bigr\|_\infty .
\end{equation}
Method \eqref{ModVer} of the class $\mathcal S^\theta$ at the truncation level $L$ is a special case of
\eqref{eq:diagmethod}: $\Omega=\{0,1,\ldots,L\}$ and $\nu_k=\nu_k^L$.

Denote by $Gf=\mathbf f$ the full set of Fourier--Hermite coefficients, and by
$C(\mathbb R)$ the set of real continuous functions on $\mathbb R$. A {\it recovery
operator} is an arbitrary, not necessarily linear, map
$\psi:\ell_{\max\{p,2\}}\to C(\mathbb R)$; its domain contains all data
$Gf+\bxi$, since $Gf\in\ell_2$ and $\bxi\in\ell_p$. The operator $\psi$ recovers the
function $f$ from the data $Gf+\bxi$ in the form $\psi(Gf+\bxi)$. The class of all such operators is denoted by
$$
\Psi=\bigl\{\psi G:\ \psi\colon\ell_{\max\{p,2\}}\to C(\mathbb R)\bigr\};
$$
the notation $\psi G$ keeps the operator $\psi$ itself on the whole data space,
since for $p>2$ the coincidence of operators on exact data does not determine their behaviour on
error sequences outside $\ell_2$.
We set
\begin{equation}\label{eq:general-error}
e_\delta(W_s^\mu,\psi G,C(\mathbb R),\ell_p)
=\sup_{f\in W_s^\mu}\sup_{\substack{\bxi\in\ell_p\\\|\bxi\|_{\ell_p}\le\delta}}
\|f-\psi(Gf+\bxi)\|_\infty,
\end{equation}
\begin{equation}\label{eq:general-minimax}
\mathcal E_\delta(W_s^\mu,\Psi,C(\mathbb R),\ell_p)
=\inf_{\psi G\in\Psi}e_\delta(W_s^\mu,\psi G,C(\mathbb R),\ell_p).
\end{equation}
Thus the comparison includes methods that use the whole sequence of
data; neither linearity nor the diagonal form of the operator is assumed.

\begin{theorem}\label{thm:diag}
Let $1\le s,p\le\infty$, $\mu>\max\{0,\bar a(s)\}$ and $0<\delta<1$.
\begin{itemize}
\item[\rm(i)] For every finite nonempty set $\Omega\subset\mathbb N_0$ and every
real sequence $\nu=(\nu_k)_{k\in\Omega}$
\begin{equation}\label{eq:diaglow}
e_\delta(W_s^\mu,\psi_\nu G_\Omega,C(\mathbb R),\ell_p)\ge c\,\delta^{e_{\mathrm{tr}}} ,
\end{equation}
where $c=c(s,p,\mu)>0$ depends neither on $\Omega$, nor on $\nu$, nor on $\card(\Omega)$.
\item[\rm(ii)] If $s\ge4/3$ or $p\le4/3$, then
$\mathcal E_\delta(W_s^\mu,\Psi,C(\mathbb R),\ell_p)\ge c\,\delta^{e_{\mathrm{tr}}}$.
If, moreover, $s\ne4/3$ and $p\ne4/3$, then
\begin{equation}\label{eq:regular-minimax}
c\,\delta^{e_{\mathrm{tr}}}
\le\mathcal E_\delta(W_s^\mu,\Psi,C(\mathbb R),\ell_p)
\le C\,\delta^{e_{\mathrm{tr}}},
\end{equation}
and for $s=4/3$ or $p=4/3$ the upper bound has the form
$\mathcal E_\delta(W_s^\mu,\Psi,C(\mathbb R),\ell_p)\le C\,\delta^{e_{\mathrm{tr}}}
\bigl(\log(\mathrm e/\delta)\bigr)^{1/4}$. Here $c,C>0$ depend only on $s,p,\mu$.
\item[\rm(iii)] Every fixed method of the class $\mathcal S^\theta$ with $\theta>\beta$ at the level
$N_*=\lceil\delta^{-1/\rho}\rceil$ has error \eqref{eq:localerror} not exceeding
$C\delta^{e_{\mathrm{tr}}}$ for $s\ne4/3$, $p\ne4/3$ and not exceeding
$C\delta^{e_{\mathrm{tr}}}(\log(\mathrm e/\delta))^{1/4}$ for $s=4/3$ or $p=4/3$, where $C$
depends also on $\theta$ and $C(\nu)$.
\end{itemize}
\end{theorem}

Thus, for $s\ne4/3$ and $p\ne4/3$ the quantity $\delta^{e_{\mathrm{tr}}}$ is the exact order
of the error in the class of finite diagonal methods, and for $s>4/3$ or $p<4/3$ also among all
recovery operators. The lower bound (i) holds for all $s,p$, and the lower bound for
$\mathcal E_\delta$ holds for $s\ge4/3$ or $p\le4/3$, including all critical
points, where the gap between the bounds does not exceed $(\log(\mathrm e/\delta))^{1/4}$ up
to a constant factor.
In the region $s<4/3<p$ no lower bound for $\mathcal E_\delta$ is established.
The proof is given in Section~\ref{diagonal}; the upper bound together with the critical
logarithmic factors is Theorem~\ref{thm:accuracy}.

\section{Analytic estimates and the accuracy of the methods $\mathcal S^\theta$}\label{accuracy}

\subsection{Pointwise estimate, power kernel and block estimate}

Throughout this section $\lambda_k=2k+1$.

\begin{lemma}[pointwise estimates]\label{sum:point}
There exist absolute constants $C\ge c>0$ with the following properties.
\begin{itemize}
\item[\rm(i)] For all $x\in\mathbb R$ and all $k\ge0$
\begin{equation}\label{sum:env}
|h_k(x)|\le C\bigl(\lambda_k^{1/3}+|x^2-\lambda_k|\bigr)^{-1/4},
\end{equation}
and in particular
\begin{equation}\label{sum:envmax}
|h_k(x)|\le C\,(k+1)^{-1/12} .
\end{equation}
\item[\rm(ii)] For all $k\ge0$
\begin{equation}\label{sum:maxlow}
\max_{x\in\mathbb R}|h_k(x)|\ge c\,(k+1)^{-1/12} .
\end{equation}
\end{itemize}
\end{lemma}

\textit{Proof.} The Hermite polynomials of \cite{Muckenhoupt1970} are defined by a
generating function and coincide with \eqref{herm} by \cite[(5.5.3), (5.5.7)]{Szego}, and the
function $\mathcal H_k$ of \cite[(2.1)]{Muckenhoupt1970} is our $h_k$. By
\cite[(2.3)]{Muckenhoupt1970}, where this estimate is derived from the table
\cite[p.~700]{AskeyWainger1965}, there exist absolute constants $C,\gamma>0$ such that
for all $k\ge0$
$$
|h_k(x)|\le C\bigl(\lambda_k^{1/3}+|x^2-\lambda_k|\bigr)^{-1/4}\ \ (x^2\le2\lambda_k),
\qquad
|h_k(x)|\le Ce^{-\gamma x^2}\ \ (x^2>2\lambda_k).
$$
In the second case $\lambda_k^{1/3}+x^2-\lambda_k\le2x^2$ and $e^{-\gamma x^2}\le C_\gamma(2x^2)^{-1/4}$;
this gives \eqref{sum:env}. Estimate \eqref{sum:envmax} follows from \eqref{sum:env}, since
$\lambda_k^{1/3}\ge(k+1)^{1/3}$. By \cite[(8.91.10), p.~242]{Szego}
$\max_{x\in\mathbb R}e^{-x^2/2}|H_k(x)|\cong(2^{k}k!)^{1/2}c_0k^{-1/12}$ with a constant $c_0>0$;
division by $(2^{k}k!\sqrt\pi)^{1/2}$ gives \eqref{sum:maxlow} for $k\ge k_0$, and the finitely many
remaining indices are absorbed into the constant, since $h_k\not\equiv0$. $\Box$

\begin{theorem}[power kernel]\label{sum:kernel}
Let $1\le q<\infty$ and
$$
K_{N,q}:=\sup_{x\in\mathbb R}\sum_{k=0}^{N}|h_k(x)|^q .
$$
Then
\begin{equation}\label{sum:law}
K_{N,q}\asymp\sum_{j=1}^{N+1}j^{-q/4}
\end{equation}
with constants depending on $q$ and not depending on $N$. Moreover,
\begin{equation}\label{sum:law_cases}
K_{N,q}\asymp
\left\{
\begin{array}{ll}
1,& q>4,\\[1mm]
\log(2+N),& q=4,\\[1mm]
(N+1)^{1-q/4},& 1\le q<4 .
\end{array}
\right.
\end{equation}
\end{theorem}

\textit{Proof.} Fix $x$. The numbers $\lambda_k$ are spaced $2$ apart, so on each
side of $x^2$ the $j$-th closest of them ($j=0,1,\ldots$) is separated from $x^2$ by at least $2j$. Hence, if
the numbers $|x^2-\lambda_k|$, $0\le k\le N$, are arranged in increasing order, then the $i$-th of them
($i=0,1,\ldots,N$) is at least $\lfloor i/2\rfloor$. Since $\lambda_k^{1/3}\ge1$, estimate
\eqref{sum:env} gives
$$
\sum_{k=0}^N|h_k(x)|^q\le C^q\sum_{i=0}^N\bigl(1+\lfloor i/2\rfloor\bigr)^{-q/4}
\le2C^q\sum_{j=1}^{N+1}j^{-q/4}
$$
uniformly in $x$. For the lower bound take $x=0$. For $q>4$ or $N\le1$ it suffices that
$K_{N,q}\ge|h_0(0)|^q=\pi^{-q/4}$. For $1\le q\le4$ and $N\ge2$, by \eqref{zero}--\eqref{wallis}
\begin{equation}\label{sum:low}
\sum_{k=0}^N|h_k(0)|^q
\ge \sum_{1\le r\le N/2}|h_{2r}(0)|^q
\gtrsim_q \sum_{1\le r\le N/2}r^{-q/4}
\asymp_q \sum_{j=1}^{N+1}j^{-q/4} .
\end{equation}
The equivalence of \eqref{sum:law} and \eqref{sum:law_cases} is a comparison of the sum with an integral.
$\Box$

For $q=2$ the upper bound in \eqref{sum:law} occurs in the proof of
\cite[Lemma~4.5]{Thangavelu1989}.

\begin{lemma}[block estimate]\label{lem:blockest}
Let $1\le s\le\infty$ and $q=s/(s-1)$. Then for all $M\in\mathbb N$ and $x\in\mathbb R$
\begin{equation}\label{eq:blockkernel}
\Bigl(\sum_{M\le k<2M}|h_k(x)|^q\Bigr)^{1/q}\le C\,M^{\bar a(s)}L_s(M),
\end{equation}
where for $s=1$ the left-hand side is $\max_{M\le k<2M}|h_k(x)|$, and $C$ depends only on $s$.
\end{lemma}

\textit{Proof.} For $s=1$, by \eqref{sum:envmax} we have
$|h_k(x)|\le CM^{-1/12}=CM^{\bar a(1)}$, since $a(1)=-\frac14<b(1)=-\frac1{12}$. Let
$s>1$, that is, $q<\infty$. On the block $\lambda_k^{1/3}\ge M^{1/3}$, and the same count as in the
proof of Theorem \ref{sum:kernel}, together with a comparison with an integral, gives
$$
\sum_{M\le k<2M}|h_k(x)|^q\le2C^q\sum_{j=0}^{M}\bigl(M^{1/3}+j\bigr)^{-q/4}
\le C\left\{
\begin{array}{ll}
M^{1-q/4},& q<4,\\[1mm]
\log(2+M),& q=4,\\[1mm]
M^{(4-q)/12},& q>4 .
\end{array}
\right.
$$
Since $qa(s)=1-q/4$ and $qb(s)=(4-q)/12$, taking the $q$-th root gives
\eqref{eq:blockkernel}: $M^{a(s)}$ for $s>4/3$, $(\log(2+M))^{1/4}$ for $s=4/3$ and
$M^{b(s)}$ for $s<4/3$. $\Box$

\subsection{Error decomposition}

We write the error of methods \eqref{ModVer} in the form
\begin{equation}\label{eq:fullError}
f(x)-S_N^\nu\mathbf f^\delta(x)=\bigl(f(x)-S_N f(x)\bigr)
+\bigl(S_N f(x)-S_N^\nu f(x)\bigr)
+\bigl(S_N^\nu f(x)-S_N^\nu\mathbf f^\delta(x)\bigr) ,
\end{equation}
where $S_N f$ is the Fourier summation method \eqref{m_F}. The truncation level $N$ in
\eqref{eq:fullError} is chosen depending on $\delta$, $p$, $s$ and $\mu$ so as to
balance the terms of the error estimate. The three differences are estimated separately below, and in the first
two the estimate \eqref{eq:blockkernel} is applied to dyadic blocks of indices.

\begin{lemma}\label{lem:tail}
Let $1\le s\le\infty$, $\mu>\max\{0,\bar a(s)\}$ and $f\in W_s^\mu$. Then the series
\eqref{function} converges absolutely and uniformly on $\mathbb R$, its sum is continuous, and for
every $N\ge1$
$$
\|f-S_N f\|_\infty\le c\|f\|_{s,\mu}N^{-(\mu-\bar a(s))}L_s(N) ,
$$
where $c$ depends only on $s$ and $\mu$.
\end{lemma}

\textit{Proof.} Let $K\ge0$, $M=K+1$ and $E_i=\{k:2^{i}M\le k<2^{i+1}M\}$, $i\ge0$;
these blocks partition $\{k>K\}$. On $E_i$ we have $\underline k=k\ge2^{i}M$, so the
$\ell_s$ norm of the coefficients $(a_k(f))_{k\in E_i}$ does not exceed $(2^{i}M)^{-\mu}\|f\|_{s,\mu}$.
H\"older's inequality at a fixed $x$ (for $s=1$ and $s=\infty$, the product of a sum and a
maximum) and \eqref{eq:blockkernel} give
$$
\sum_{k\in E_i}|a_k(f)h_k(x)|\le C\|f\|_{s,\mu}(2^{i}M)^{-(\mu-\bar a(s))}L_s(2^{i}M) .
$$
For $s\ne4/3$ the sum over $i$ is a geometric progression with ratio
$2^{-(\mu-\bar a(s))}<1$. For $s=4/3$ we have $\bar a(s)=0$ and $2+2^{i}M\le(2+M)^{i+1}$, so that
$L_s(2^{i}M)\le(i+1)^{1/4}L_s(M)$, and the series $\sum_i(i+1)^{1/4}2^{-i\mu}$ converges. In both cases
\begin{equation}\label{eq:tailsum}
\sup_{x\in\mathbb R}\sum_{k>K}|a_k(f)h_k(x)|\le C\|f\|_{s,\mu}(K+1)^{-(\mu-\bar a(s))}L_s(K+1).
\end{equation}
The right-hand side tends to zero as $K\to\infty$, so the series \eqref{function} converges
absolutely and uniformly, and its sum is continuous as a uniform limit of continuous partial
sums. For $K=N$ estimate \eqref{eq:tailsum} gives the statement of the lemma, since
$(N+1)^{-(\mu-\bar a(s))}\le N^{-(\mu-\bar a(s))}$ and $L_s(N+1)\le2L_s(N)$. $\Box$

\begin{lemma}\label{lem:smooth}
Let $1\le s\le\infty$, $\mu>0$ and $f\in W_s^\mu$. Then for every method of
$\mathcal S^\theta$ with $\theta>\mu-\bar a(s)$ and for every $N\ge1$
$$
\|S_N f-S_N^\nu f\|_\infty
\le cC(\nu)\|f\|_{s,\mu}N^{-(\mu-\bar a(s))}L_s(N) ,
$$
where $c$ depends only on $s$, $\mu$ and $\theta$.
\end{lemma}

\textit{Proof.} For $k=0$ condition \eqref{qual} gives $\nu_0^N=1$, and the term with
$k=0$ in the difference $S_N f-S_N^\nu f=\sum_{k=0}^N(1-\nu_k^N)a_k(f)h_k$ vanishes. Let
$I=\lfloor\log_2N\rfloor$ and $B_i=\{k\le N:2^i\le k<2^{i+1}\}$, $i=0,\ldots,I$. On $B_i$ we have
$\underline k=k\ge2^i$ and $|1-\nu_k^N|\le C(\nu)2^\theta2^{i\theta}N^{-\theta}$. As in the
proof of Lemma \ref{lem:tail}, H\"older's inequality and \eqref{eq:blockkernel} with $M=2^i$
give
$$
\sum_{k\in B_i}|(1-\nu_k^N)a_k(f)h_k(x)|
\le CC(\nu)\|f\|_{s,\mu}N^{-\theta}2^{it}L_s(2^i),
\qquad t=\theta-\mu+\bar a(s)>0 .
$$
Since $L_s(2^i)\le L_s(N)$ and $\sum_{i\le I}2^{it}\le(1-2^{-t})^{-1}2^{It}\le CN^t$,
we obtain $\|S_N f-S_N^\nu f\|_\infty\le cC(\nu)\|f\|_{s,\mu}N^{t-\theta}L_s(N)$, as
required. $\Box$

\begin{lemma}\label{lem:noise}
Let $1\le p\le\infty$, $p'=p/(p-1)$, $f\in {\cal L}_2(\mathbb R)$, and let
\eqref{perturbation} hold. Then for every method of $\mathcal S^\theta$ and for every $N\ge1$
\begin{equation}\label{eq:sigma}
\|S_N^\nu f-S_N^\nu\mathbf f^\delta\|_\infty
\le\bigl(1+C(\nu)\bigr)\delta\,K_{N,p'}^{1/p'}
\le c\,\delta N^{a_+(p)}L_p(N) ,
\end{equation}
where for $p=1$ the quantity $K_{N,\infty}^{1/\infty}$ is replaced by $\max_{k\le N}\|h_k\|_\infty$, and $c$
depends only on $p$ and $C(\nu)$.
\end{lemma}

\textit{Proof.} For $0\le k\le N$ condition \eqref{qual} gives
$|1-\nu_k^N|\le C(\nu)(k/N)^\theta\le C(\nu)$, whence $|\nu_k^N|\le1+C(\nu)$. The right-hand side of
\eqref{ModVer} is linear in the input sequence, so
$S_N^\nu f-S_N^\nu\mathbf f^\delta=-\sum_{k=0}^N \nu_k^N \xi_k h_k$, and at a fixed
$x$ H\"older's inequality gives
$$
\Bigl|\sum_{k=0}^N \nu_k^N \xi_k h_k(x)\Bigr|
\le\bigl(1+C(\nu)\bigr)\|\bxi\|_{\ell_p}\Bigl(\sum_{k=0}^N|h_k(x)|^{p'}\Bigr)^{1/p'} ;
$$
the supremum over $x$ gives the first inequality in \eqref{eq:sigma}. The second inequality follows from Theorem
\ref{sum:kernel} with $q=p'$: for $p<4/3$ we have $p'>4$ and $K_{N,p'}\asymp1$, for $p=4/3$,
$K_{N,4}\asymp\log(2+N)$, for $p>4/3$, $K_{N,p'}^{1/p'}\asymp(N+1)^{3/4-1/p}
\le2^{a(p)}N^{a(p)}$; for $p=1$ estimate \eqref{sum:envmax} suffices. $\Box$

\begin{theorem}\label{thm:accuracy}
Let $1\le s\le\infty$, $1\le p\le\infty$, $\mu>\max\{0,\bar a(s)\}$, let
\eqref{perturbation} hold, and let a method of $\mathcal S^\theta$ have $\theta>\mu-\bar a(s)$.
Then for every $f\in W_s^\mu$ and every $N\ge1$
\begin{equation}\label{eq:acc}
\|f-S_N^\nu\mathbf f^\delta\|_\infty
\le c\bigl(N^{-(\mu-\bar a(s))}L_s(N)+\delta N^{a_+(p)}L_p(N)\bigr) ,
\end{equation}
and if $\gamma_1 \delta^{-1/\rho}\le N\le\gamma_2 \delta^{-1/\rho}$ for some
$0<\gamma_1 \le\gamma_2$, then
$$
\|f-S_N^\nu\mathbf f^\delta\|_\infty\le c\,\delta^{e_{\mathrm{tr}}}
\qquad(s\ne4/3,\ p\ne4/3),
$$
$$
\|f-S_N^\nu\mathbf f^\delta\|_\infty
\le c\,\delta^{e_{\mathrm{tr}}}\bigl(\log(\mathrm e/\delta)\bigr)^{1/4}
\qquad(s=4/3\ \text{or}\ p=4/3) .
$$
The constants depend only on $s$, $p$, $\mu$, $\theta$, $C(\nu)$, $\gamma_1$ and $\gamma_2$. In
particular, the level $N_*=\lceil\delta^{-1/\rho}\rceil$ satisfies this condition with
$\gamma_1=1$, $\gamma_2=2$, and the method uses $N_*+1\le3\delta^{-1/\rho}$ perturbed
Fourier--Hermite coefficients.
\end{theorem}

\textit{Proof.} By Lemmas \ref{lem:tail}, \ref{lem:smooth} and \ref{lem:noise},
\eqref{eq:fullError} gives \eqref{eq:acc} for every $N\ge1$ and every
$0<\delta<1$.

Now let $\gamma_1 \delta^{-1/\rho}\le\allowbreak N\le\allowbreak\gamma_2 \delta^{-1/\rho}$. The left
inequality gives $N^{-(\mu-\bar a(s))}\le\gamma_1^{-(\mu-\bar a(s))}\delta^{e_{\mathrm{tr}}}$, and the right one gives
$\delta N^{a_+(p)}\le\gamma_2^{a_+(p)}\delta^{1-a_+(p)/\rho}
=\gamma_2^{a_+(p)}\delta^{e_{\mathrm{tr}}}$, since
$$
1-\frac{a_+(p)}{\rho}=\frac{\rho-a_+(p)}{\rho}
=\frac{\mu-\bar a(s)}{\rho}=e_{\mathrm{tr}} .
$$
This equality means that with this choice of $N$ both terms in \eqref{eq:acc} have the
same order $\delta^{e_{\mathrm{tr}}}$. For $s\ne4/3$ and $p\ne4/3$ all logarithmic factors
equal one. Let $s=4/3$ or $p=4/3$. Since $\delta^{-1/\rho}>1$, we have
$2+N\le(2+\gamma_2)\delta^{-1/\rho}$, whence
$\log(2+N)\le C'\log(\mathrm e/\delta)$ with $C'=\max\{1,\log(2+\gamma_2),1/\rho\}$. Therefore
$L_s(N),L_p(N)\le(C')^{1/4}(\log(\mathrm e/\delta))^{1/4}$, and each term in
\eqref{eq:acc} carries at most the power $1/4$ of the logarithm.

Finally, $t=\delta^{-1/\rho}>1$, so $t\le N_*\le t+1\le2t$ and $N_*+1\le2t+1\le3t$;
method \eqref{ModVer} uses the data only for $k=0,1,\ldots,N_*$. $\Box$

\vskip 2mm

\begin{remark}\label{rem:technique}\rm
The upper bounds of Lemmas \ref{lem:tail}, \ref{lem:smooth}, \ref{lem:noise} and of Theorem
\ref{thm:accuracy} are obtained not by the standard scheme used, for example, in
\cite[Theorem~3.1]{MathePereverzev2002} (for $s=p=2$ for systems with
$\|\varphi_k\|_\infty\asymp k^{\beta}$, $\beta\ge0$) and in \cite{Sol_Stas_UMZ2022},
\cite{SemSol_Sim}, \cite{SolodkyVolynets2026}.
In the standard scheme each term is first bounded
by the maximum of its Hermite function, $|h_k(x)|\le\|h_k\|_\infty$, and only then summed over $k$ by
means of H\"older's inequality. Here H\"older's inequality is applied at a fixed $x$
to the whole block of indices, and the supremum over $x$ is taken after the summation: the quantities
\eqref{sum:law}, \eqref{eq:blockkernel} and \eqref{eq:sigma} are computed for
the block as a whole. Already for unit weights this changes the order for $1\le q\le12$ (for $q>12$
both sums below are bounded): by Lemma \ref{sum:point} and Theorem \ref{sum:kernel}, for
$1\le q<\infty$
$$
\sum_{k=0}^N\|h_k\|_\infty^q\asymp\sum_{j=1}^{N+1}j^{-q/12},
\qquad
\sup_{x\in\mathbb R}\sum_{k=0}^N|h_k(x)|^q\asymp\sum_{j=1}^{N+1}j^{-q/4},
$$
since the maximum of the function $h_k$ is attained near the point
$x_k\approx\sqrt{2k+1}$: outside the strip $|x^2-2k-1|\le A(2k+1)^{1/3}$ the bound \eqref{sum:env}
is smaller than \eqref{sum:maxlow} for a sufficiently large absolute constant $A$.

The standard scheme leads to the following estimates. Set
$$
\tilde\beta=\mu-\Bigl(\frac{11}{12}-\frac1s\Bigr),
\qquad
\tilde a(p)=\max\Bigl\{\frac{11}{12}-\frac1p,\,0\Bigr\},
\qquad
\tilde e=\frac{\tilde\beta}{\tilde\beta+\tilde a(p)} .
$$
For $\mu>\max\{0,11/12-1/s\}$ and a method of $\mathcal S^\theta$ with $\theta>\tilde\beta$ it gives, for
every $N\ge1$,
$$
\|f-S_N^\nu\mathbf f^\delta\|_\infty
\le c\bigl(N^{-\tilde\beta}+\delta N^{\tilde a(p)}\tilde L_p(N)\bigr),
$$
where $\tilde L_p(N)=(\log(2+N))^{1/12}$ for $p=12/11$ and $\tilde L_p(N)=1$ for $p\ne12/11$, and,
for $N\asymp\delta^{-1/(\tilde\beta+\tilde a(p))}$, the estimate $c\,\delta^{\tilde e}$ for
$p\ne12/11$ and $c\,\delta(\log(\mathrm e/\delta))^{1/12}$ for $p=12/11$. Here
$\tilde e=(\mu+1/s-11/12)/(\mu+1/s-1/p)$ for $p>12/11$ and $\tilde e=1$ for $p<12/11$.

Let us compare these estimates with Theorem \ref{thm:accuracy}. By \eqref{exps} we have
$$
\frac{11}{12}-\frac1s-\bar a(s)=
\left\{
\begin{array}{ll}
\frac16, & s>4/3,\\[1mm]
\frac23\bigl(1-\frac1s\bigr), & s<4/3,
\end{array}
\right.
\qquad
\tilde a(p)-a_+(p)=
\left\{
\begin{array}{ll}
\frac16, & p>4/3,\\[1mm]
\frac{11}{12}-\frac1p, & 12/11<p<4/3,\\[1mm]
0, & p\le12/11 ,
\end{array}
\right.
$$
hence $\tilde\beta\le\beta$ and $\tilde a(p)\ge a_+(p)$. Moreover, the condition
$\mu>\max\{0,11/12-1/s\}$ implies $\mu>\max\{0,\bar a(s)\}$, and every method with $\theta>\beta$
also satisfies the condition $\theta>\tilde\beta$, so that for $\mu>\max\{0,11/12-1/s\}$ and
$\theta>\beta$ both estimates
refer to the same method. The comparison below is made for $s\ne4/3$ and $p\ne4/3$, where
the estimate of Theorem \ref{thm:accuracy} contains no factors $L_s$ and $L_p$. Since $t/(t+a)$ increases in $t>0$ and decreases in
$a\ge0$, we have $\tilde e\le e_{\mathrm{tr}}$, the inequality being strict for $p>12/11$, while for
$p=12/11$ the standard scheme loses a logarithmic factor. For $p<12/11$ both estimates
have order $\delta$, but the standard scheme needs $\asymp\delta^{-1/\tilde\beta}$
coefficients, while the method of Theorem \ref{thm:accuracy} with the level $N_*$ needs $\asymp\delta^{-1/\beta}$; these orders
coincide only for $s=1$.

By Theorem \ref{thm:diag}\,(i), (iii) the order $\delta^{e_{\mathrm{tr}}}$ is not improved by any finite
diagonal method, and the methods $\mathcal S^\theta$ with $\theta>\beta$ attain it. Therefore
the difference of the exponents is a loss caused by the standard estimation scheme, and not a property of the method.
For comparison, in
\cite{SolodkyVolynets2026} the same scheme gave the exact order (ibid., Lemma~2.1 and
Theorem~4.7), since for the Laguerre functions the sum of a block of $N$ functions with indices from
$[N,3N]$ has order $N$ in the weighted uniform norm, that is, the order of the sum of its
maxima: the lower bound is given in Lemma~4.1 of the same paper, and the upper bound $CN$ follows
from the uniform boundedness of the Laguerre functions (ibid., (2.4)). For the Hermite functions the sum of maxima over such a block
$\Omega$ has order $N^{11/12}$, whereas by Theorem \ref{sum:kernel} with $q=1$ and
Lemma \ref{lem:block}\,{\rm(i)}
$$
\max_{\varepsilon_k=\pm1}\Bigl\|\sum_{k\in\Omega}\varepsilon_k h_k\Bigr\|_\infty\asymp N^{3/4} .
$$
\end{remark}

\section{Lower bounds and the exact order of diagonal methods}\label{diagonal}

In this section Theorem~\ref{thm:diag} is proved. The main tool is Lemma~\ref{lem:block}
on lower bounds for sums of Hermite functions with indices from the interval $[N,3N]$. With its help
the lower bounds of Theorem~\ref{thm:diag} are obtained: for every finite diagonal method,
for all $s$ and $p$; for all recovery operators, for $s\ge4/3$ or $p\le4/3$.
Together with the upper bounds of Theorem~\ref{thm:accuracy}, for $s\ne4/3$ and $p\ne4/3$ they give the
exact order $\delta^{e_{\mathrm{tr}}}$ of the error of finite diagonal methods and, if
in addition $s>4/3$ or $p<4/3$, also the exact order of the smallest error among all
recovery operators.

Let, as before, $1\le s,p\le\infty$. The exponents are defined in \eqref{exps} and
\eqref{beta}. For a coefficient sequence we use the same norm
$\|\cdot\|_{s,\mu}$ as in \eqref{norm}.

\begin{lemma}\label{lem:block}
There exist absolute constants $\kappa_0=19/50$, $\kappa\in(0,1)$, $c_1>0$ and an integer $N_0$ such that
for every integer $N\ge N_0$ and every set $\Omega$ of $N$ distinct integers of the
interval $[N,3N]$ the following holds.
\begin{itemize}
\item[\rm(i)] At the point $x_0=\kappa_0 N^{-1/2}$, for all integers $k\in[N,3N]$,
\begin{equation}\label{bulkpoint}
|h_k(x_0)|\ge\frac1{4\sqrt\pi}\,k^{-1/4} ,
\end{equation}
and therefore, for $\varepsilon_k=\operatorname{sign}h_k(x_0)$,
\begin{equation}\label{bulksum}
\Bigl\|\sum_{k\in\Omega}\varepsilon_k h_k \Bigr\|_\infty
\ge\sum_{k\in\Omega}|h_k(x_0)|
\ge\frac{3^{-1/4}}{4\sqrt\pi}\,N^{3/4} .
\end{equation}
\item[\rm(ii)] There exist a set $W\subset\Omega$ with $\card(W)\ge\kappa N^{1/3}$ and a point
$x_W$ such that $|h_k(x_W)|\ge c_1N^{-1/12}$ for all $k\in W$.
\item[\rm(iii)] For any $1\le s\le\infty$ and $\mu>0$ the function
$$
f=(3N)^{-\mu}N^{-1/s}\sum_{k\in\Omega}\varepsilon_k h_k
$$
with the signs from {\rm(i)} satisfies $\|f\|_{s,\mu}\le1$, its coefficients $a_k(f)$ vanish
outside $\Omega$, and
\begin{equation}\label{blocktail}
\|f\|_\infty\ge\tfrac1{10}\,3^{-\mu}N^{-\mu+3/4-1/s} .
\end{equation}
\end{itemize}
\end{lemma}

\textit{Proof.} Items (i) and (ii) are derived from two asymptotic formulas of Szeg\"{o}: in
the oscillatory region for (i) and near the turning point for (ii). At the chosen point the main
term of the formula is bounded away from zero in absolute value, and the remainder term is small for $N\ge N_0$; item (iii)
follows by direct substitution.

Set $\lambda_k=2k+1$. Formula (8.22.8) of \cite{Szego}, after
division by $(2^k k!\sqrt\pi)^{1/2}$, gives
\begin{equation}\label{pr}
h_k(x)=A_k \Bigl\{\cos\Bigl(\lambda_k^{1/2}x-\frac{k\pi}2\Bigr)
+\frac{x^3}6\,\lambda_k^{-1/2}\sin\Bigl(\lambda_k^{1/2}x-\frac{k\pi}2\Bigr)+R_k(x)\Bigr\},
\end{equation}
where $A_k=\Gamma(k+1)\Gamma(k/2+1)^{-1}(2^k k!\sqrt\pi)^{-1/2}$ and $|R_k(x)|\le Bk^{-1}$ uniformly on the interval $[0,1]$ with an absolute constant $B$.
The duplication formula $\Gamma(k+1)=\pi^{-1/2}2^k \Gamma((k+1)/2)\Gamma(k/2+1)$ gives
$A_k^2=\pi^{-1}\Gamma((k+1)/2)/\Gamma(k/2+1)$, and the log-convexity of the gamma function gives
$\Gamma(k/2+1)^2\le\Gamma((k+1)/2)\,\Gamma(k/2+3/2)=\frac{k+1}2\,\Gamma((k+1)/2)^2$. Therefore
for every $k\ge1$ we have $A_k\ge\pi^{-1/2}((k+1)/2)^{-1/4}\ge\pi^{-1/2}k^{-1/4}$.

We prove (i). Let $k\in[N,3N]$ be an integer. The point $x_0\le19/50$ belongs to $[0,1]$, and the quantity
$\varphi_k=\lambda_k^{1/2}x_0=\kappa_0((2k+1)/N)^{1/2}$, by $2<(2k+1)/N\le7$, satisfies
$\pi/6<\kappa_0\sqrt2=0.537\ldots<\varphi_k\le\kappa_0\sqrt7=1.005\ldots<\pi/3$.
For $k=2m$ we have $\cos(\varphi_k-k\pi/2)=(-1)^m\cos\varphi_k$, and for $k=2m+1$ it equals
$(-1)^m\sin\varphi_k$. On $[\pi/6,\pi/3]$ both functions $\cos t$ and $\sin t$ are at least $1/2$;
therefore the main term in the braces of \eqref{pr} is at least $1/2$ in absolute value. Since $x_0\le19/50$ and
$\lambda_k\ge1$, the second term does not exceed $\frac16(19/50)^3<10^{-2}$, and for $N\ge5B$ we have
$|R_k(x_0)|\le B/N\le1/5$. Consequently, $|h_k(x_0)|\ge A_k(\frac12-10^{-2}-\frac15)>A_k/4$,
and the estimate of $A_k$ gives \eqref{bulkpoint}. For $\varepsilon_k=\operatorname{sign}h_k(x_0)$ all
terms $\varepsilon_k h_k(x_0)=|h_k(x_0)|$ are nonnegative, and the norm bounds the value at a
point from above. Since $\Omega$ consists of $N$ indices not exceeding $3N$, we obtain \eqref{bulksum}.

We prove (ii). Theorem~8.22.9 of \cite{Szego} in the normalization \eqref{herm} gives
\begin{equation}\label{airyexp}
h_k(x)=\gamma_{\mathcal A}\,k^{-1/12}\bigl\{\mathcal A(t)+O(k^{-2/3})\bigr\},
\qquad
x=\lambda_k^{1/2}-2^{-1/2}3^{-1/3}k^{-1/6}t .
\end{equation}
Here $t=2^{1/2}3^{1/3}k^{1/6}(\lambda_k^{1/2}-x)$ is the parameter that gives the position of $x$
relative to the point $\lambda_k^{1/2}$. The formula is uniform in $|t|\le1$, where
$\gamma_{\mathcal A}>0$ is an absolute constant; the constant in the remainder
is denoted by $B'$. The function $\mathcal A(t)$ is the Airy function of Section~1.81 of the same book. It
is continuous, and by its power series $\mathcal A(0)=\pi/(3\Gamma(2/3))>0$. By continuity
there is $\eta>0$ such that $|\mathcal A(t)-\mathcal A(0)|<\mathcal A(0)/2$ for $|t|<\eta$.
Taking $t_0=\min\{1,\eta/2\}$, we obtain $\mathcal A(t)\ge\mathcal A(0)/2$ for $0\le t\le t_0$;
the number $t_0$ is absolute, since the function $\mathcal A$ is fixed.

Set $\tau=t_0/\sqrt3$ and $J=\lfloor\tau N^{1/3}\rfloor$. For $N\ge\tau^{-3}$ we have
$u=\tau N^{1/3}\ge1$, so $J\ge1$ and $J\ge u/2$ (for $1\le u<2$ this is obvious,
and for $u\ge2$ it follows from $\lfloor u\rfloor>u-1\ge u/2$). Moreover,
$\tau\le1/\sqrt3$ and $N\ge6$, whence $J\le\tau N^{1/3}\le N/2\le N-1$. We split the $2N+1$ integers of the interval $[N,3N]$ into
consecutive intervals of at most $J$ numbers each; there are at most
$(2N+1)/J+1\le3N/J$ of them, so one of them contains at least $J/3\ge\kappa N^{1/3}$ points of
$\Omega$, where $\kappa=\tau/6$. Denote the set of these points by $W$, and let
$k_*=\min W$; then $0\le k-k_*<J\le\tau N^{1/3}$ for $k\in W$.

Set $x_W=\lambda_{k_*}^{1/2}$. Substituting $x=x_W$ into the expression for $t$ after
\eqref{airyexp}, we obtain $t_k=2^{1/2}3^{1/3}k^{1/6}(\lambda_k^{1/2}-\lambda_{k_*}^{1/2})$.
Since $k\ge k_*$ and $\lambda_k=2k+1$, we have $t_k\ge0$. From
$\lambda_k^{1/2}+\lambda_{k_*}^{1/2}>2(2N)^{1/2}$ and $k\le3N$ it follows that
$$
t_k=2^{1/2}3^{1/3}k^{1/6}\,\frac{2(k-k_*)}{\lambda_k^{1/2}+\lambda_{k_*}^{1/2}}
<2^{1/2}3^{1/3}(3N)^{1/6}\,\frac{\tau N^{1/3}}{(2N)^{1/2}}
=\sqrt3\,\tau=t_0 .
$$
Thus $0\le t_k<t_0\le1$, and to $t=t_k$ the estimate
$\mathcal A(t)\ge\mathcal A(0)/2$ chosen above applies. The remainder in \eqref{airyexp} does not exceed
$B'k^{-2/3}\le\mathcal A(0)/4$ for $k\ge(4B'/\mathcal A(0))^{3/2}$. Hence for $k\in W$,
taking into account $k\le3N$,
$$
|h_k(x_W)|\ge\tfrac14\gamma_{\mathcal A}\mathcal A(0)\,k^{-1/12}\ge c_1N^{-1/12},
\qquad c_1=\tfrac14\,3^{-1/12}\gamma_{\mathcal A}\mathcal A(0) .
$$

We prove (iii). By the orthonormality of the system $\{h_k\}$ the coefficients of the function $f$ from (iii)
are equal to $(3N)^{-\mu}N^{-1/s}\varepsilon_k$ on $\Omega$ and to zero outside $\Omega$.
On $\Omega$ we have $\underline k=k\le3N$, and for $\mu>0$ this gives
$\underline k^{\,\mu}|a_k(f)|\le N^{-1/s}$; therefore $\|f\|_{s,\mu}\le(N\cdot N^{-1})^{1/s}=1$
for $s<\infty$ and $\|f\|_{\infty,\mu}\le1$. The norm bounds the value at the point $x_0$ from above, and by
\eqref{bulksum}
$$
\|f\|_\infty\ge(3N)^{-\mu}N^{-1/s}\sum_{k\in\Omega}|h_k(x_0)|
\ge\frac{3^{-1/4}}{4\sqrt\pi}\,3^{-\mu}N^{-\mu+3/4-1/s},
$$
which gives \eqref{blocktail}, since $3^{-1/4}(4\sqrt\pi)^{-1}=0.107\ldots$

We add that all restrictions on $N$ are satisfied for $N\ge N_0$, where
$$
N_0=\lceil\max\{5B,\ \tau^{-3},\ (4B'/\mathcal A(0))^{3/2}\}\rceil .
$$
This is an absolute number, since
the constants $B$ and $B'$ of the remainder terms in \eqref{pr} and \eqref{airyexp} and the number $t_0$
are absolute.
$\Box$

\subsection{Proof of Theorem \ref{thm:diag}}

\textit{Two-point estimate.} Let $g\in W_s^\mu$ and $\|Gg\|_{\ell_p}\le\delta$. The functions $g$
and $-g$ correspond to the admissible errors $-Gg$ and $Gg$, which give the same zero
data. If $u=\psi(0)$, then the triangle inequality gives
\[
\max\{\|g-u\|_\infty,\|-g-u\|_\infty\}\ge\|g\|_\infty.
\]
Consequently, the error \eqref{eq:general-error} of every operator $\psi G\in\Psi$ is at
least $\|g\|_\infty$. This argument uses neither the linearity of the operator nor a restriction on
the number of data.

\textit{The threshold in $\delta$.} Let $N_0$ be from Lemma~\ref{lem:block} and
$\delta_0=N_0^{-\rho}\in(0,1]$; for $0<\delta<\delta_0$ we have
$\lceil\delta^{-1/\rho}\rceil\ge N_0$. It suffices to prove the lower bounds in (i) and (ii) for
$0<\delta<\delta_0$. Indeed, the quantities \eqref{eq:localerror} and \eqref{eq:general-error}
do not decrease in $\delta$, since the supremum is taken over a larger ball of errors, and hence the
infimum \eqref{eq:general-minimax} does not decrease either. If the estimate $e_\delta\ge c\,\delta^{e_{\mathrm{tr}}}$
is proved for $\delta<\delta_0$, then for $\delta_0\le\delta<1$
$$
e_\delta\ge e_{\delta_0/2}\ge c\,(\delta_0/2)^{e_{\mathrm{tr}}}
\ge\tfrac c2\,N_0^{-\beta}\,\delta^{e_{\mathrm{tr}}},
$$
since $e_{\mathrm{tr}}\le1$, $\delta_0^{e_{\mathrm{tr}}}=N_0^{-\beta}$ and
$\delta^{e_{\mathrm{tr}}}<1$; the same holds for $\mathcal E_\delta$. The number $N_0$ is absolute, so
the new constant depends only on $s,p,\mu$.

\textit{The lower bound in} (ii). For $p\le4/3$ take $g=\delta h_0$. Then
$\|g\|_{s,\mu}=\delta<1$, $\|Gg\|_{\ell_p}=\delta$, and
$\|g\|_\infty=\pi^{-1/4}\delta=\pi^{-1/4}\delta^{e_{\mathrm{tr}}}$, since $a_+(p)=0$ and
$e_{\mathrm{tr}}=1$.

Let $s\ge4/3$ and $p>4/3$. Here $\bar a(s)=a(s)$ (for $s=4/3$ both quantities are zero),
$\beta=\mu-a(s)>0$ and $\rho=\mu+1/s-1/p>0$. Set $N=\lceil\delta^{-1/\rho}\rceil$ and
$\Lambda=\{N,N+1,\ldots,2N-1\}$. Since $\delta<\delta_0$, we have $N\ge N_0$, and to these $N$
indices Lemma~\ref{lem:block} applies. Let $g$ be the function of its item~{\rm(iii)}:
\[
g=(3N)^{-\mu}N^{-1/s}\sum_{k\in\Lambda}\varepsilon_k h_k,
\qquad \varepsilon_k=\operatorname{sign}h_k(x_0).
\]
By Lemma~\ref{lem:block}\,{\rm(iii)} $\|g\|_{s,\mu}\le1$ and
\[
\|g\|_\infty\ge\tfrac1{10}\,3^{-\mu}N^{-\mu+3/4-1/s}
=\tfrac1{10}\,3^{-\mu}N^{-\beta}
\ge\tfrac1{10}\,3^{-\mu}2^{-\beta}\delta^{\beta/\rho},
\]
since $\delta^{-1/\rho}\le N\le2\delta^{-1/\rho}$. Moreover, by orthonormality
\[
\|Gg\|_{\ell_p}=3^{-\mu}N^{-\mu-1/s+1/p}=3^{-\mu}N^{-\rho}\le\delta .
\]
The two-point estimate and passing to the
infimum over $\Psi$ give the lower bound in \eqref{eq:regular-minimax}.

\textit{Item} (i) \textit{for $s\ge4/3$ or $p\le4/3$.} Let $\Omega$ and $\nu$ be from (i). The operator $\psi=\psi_\nu\circ P_\Omega$
is defined on the whole of $\ell_{\max\{p,2\}}$ and takes values in $C(\mathbb R)$ as a finite
sum of Hermite functions, so that $\psi G\in\Psi$ and
$\psi(Gf+\bxi)=\psi_\nu(G_\Omega f+P_\Omega\bxi)$. Therefore the error \eqref{eq:localerror}
equals $e_\delta(W_s^\mu,\psi G,C(\mathbb R),\ell_p)\ge
\mathcal E_\delta(W_s^\mu,\Psi,C(\mathbb R),\ell_p)$, and \eqref{eq:diaglow} follows from the already
proved lower bound in (ii) with the same constant $c(s,p,\mu)$.

\textit{Item} (i) \textit{for $s<4/3<p$.} Here $\bar a(s)=b(s)$, $\beta=\mu-b(s)$ and
$a_+(p)=a(p)>0$. We extend the factors by zero: $\nu_k=0$ for $k\notin\Omega$. Then
the method multiplies the $k$-th data coefficient by $\nu_k$ for all $k\ge0$; for an
unused coefficient $\nu_k=0$ and $|1-\nu_k|=1$, and to the data $Gf+\bxi$ the method
assigns the function $\sum_{k\ge0}\nu_k \left(a_k(f)+\xi_k \right)h_k$.

Set
$$
M=\lceil\delta^{-1/\rho}\rceil,
\qquad
\Omega_M=\{M,M+1,\ldots,3M\} .
$$
Since $\delta<\delta_0$, we have $M\ge N_0$; moreover, $\delta^{-1/\rho}\le M\le2\delta^{-1/\rho}$.
The set $\Omega_M$ contains the $2M+1$ integer points of the interval $[M,3M]$. Let
$x_0=\kappa_0 M^{-1/2}$ be the point from item~{\rm(i)} of Lemma \ref{lem:block}. Set
$$
\mathcal K=\left\{k\in\Omega_M:\ |1-\nu_k|<\tfrac12\right\} .
$$
Every $k\in\mathcal K$ satisfies $|\nu_k|>1/2$, in particular $\nu_k \ne0$, and therefore
lies in $\Omega$ and is used by the method. Every $k\in\Omega_M \setminus\mathcal K$ satisfies
$|1-\nu_k|\ge1/2$. We consider two cases.

\textit{Case} (a): $\card(\Omega_M \setminus\mathcal K)\ge M$. Choose
$\Lambda\subset\Omega_M \setminus\mathcal K$ with $\card(\Lambda)=M$; Lemma \ref{lem:block}
applies to $\Lambda$. Take $\bxi=\mathbf 0$, the subset $W\subset\Lambda$ and the point $x_W$
from Lemma \ref{lem:block}, for which $\card(W)\ge\kappa M^{1/3}$ and
$|h_k(x_W)|\ge c_1M^{-1/12}$ for $k\in W$. Set
$a_k=(3M)^{-\mu}\card(W)^{-1/s}\sigma_k$ for $k\in W$ with
$\sigma_k=\mathrm{sign}\bigl((1-\nu_k)h_k(x_W)\bigr)$, $a_k=0$ outside $W$, and
$f=\sum_{k}a_k h_k$. On $W$ we have $\underline k^{\,\mu}|a_k|\le\card(W)^{-1/s}$, therefore
$\|f\|_{s,\mu}\le1$. The error of the method on these data is
$\bigl\|\sum_k(1-\nu_k)a_k h_k \bigr\|_\infty$, and at the point $x_W$ all terms of this sum
are nonnegative by the choice of signs, so that
$$
e_\delta(W_s^\mu,\psi_\nu G_\Omega,C(\mathbb R),\ell_p)
\ge\sum_{k\in W}|1-\nu_k|\,|a_k|\,|h_k(x_W)|
\ge c\,M^{-\mu}\card(W)^{1-1/s}M^{-1/12}
$$
$$
\ge c\,M^{-\mu+1/3-1/(3s)-1/12}
= c\,M^{-\beta},
$$
where we used $1-1/s\ge0$, $\card(W)\ge\kappa M^{1/3}$ and the equality
$1/3-1/(3s)-1/12=b(s)$; recall that $b(s)=1/4-1/(3s)$ by \eqref{exps}, and
$\beta=\mu-\bar a(s)=\mu-b(s)$ by \eqref{beta}, since $\bar a(s)=b(s)$ for $s<4/3$.

\textit{Case} (b): $\card(\Omega_M \setminus\mathcal K)\le M-1$. Then
$\card(\mathcal K)\ge(2M+1)-(M-1)=M+2$. Take $f=0$ and the error
$$
\xi_k=\delta\,\card(\mathcal K)^{-1/p}\,\mathrm{sign}\bigl(\nu_k h_k(x_0)\bigr),
\quad k\in\mathcal K,
\qquad
\xi_k=0,\quad k\notin\mathcal K ,
$$
so that $\|\bxi\|_{\ell_p}=\delta$, and for $p=\infty$ here $\card(\mathcal K)^{-1/p}=1$.
Every $k\in\mathcal K$ is used by the method, to these data the method assigns the function
$\sum_{k\in\mathcal K}\nu_k \xi_k h_k$, and at the point $x_0$ all terms are nonnegative.
By Lemma \ref{lem:block} $|h_k(x_0)|\ge c\,k^{-1/4}\ge c\,(3M)^{-1/4}$ for $k\in[M,3M]$,
therefore
$$
e_\delta(W_s^\mu,\psi_\nu G_\Omega,C(\mathbb R),\ell_p)
\ge\delta\,\card(\mathcal K)^{-1/p}\sum_{k\in\mathcal K}|\nu_k|\,|h_k(x_0)|
$$
$$
\ge c\,\delta\,\card(\mathcal K)^{1-1/p}M^{-1/4}
\ge c\,\delta\,M^{a(p)} ,
$$
where we used $|\nu_k|>1/2$ on $\mathcal K$, $\card(\mathcal K)\ge M$ and
$1-1/p\ge0$; recall that $a(p)=3/4-1/p$ by \eqref{exps}.

It remains to balance. From $M\ge\delta^{-1/\rho}$ it follows that
$\delta M^{a(p)}\ge\delta^{1-a(p)/\rho}=\delta^{e_{\mathrm{tr}}}$, which gives
\eqref{eq:diaglow} in case (b). From $M\le2\delta^{-1/\rho}$ it follows that
$M^{-\beta}\ge2^{-\beta}\delta^{e_{\mathrm{tr}}}$, which gives \eqref{eq:diaglow} in case (a).
The quantity $M$ is defined by $\delta$, $s$, $p$, $\mu$ before $\Omega$ and $\nu$ are named, and
the other factors are $3^{-\mu}$, $\kappa^{1-1/s}$, $2^{-\beta}$, $3^{-1/4}$, the absolute
constants of Lemma \ref{lem:block} and the number $1/2$ from the definition of the set $\mathcal K$. Therefore the constant in \eqref{eq:diaglow}
depends neither on $\Omega$, nor on $\nu$, nor on $\card(\Omega)$; about the factors we used
only that they are real.

\textit{Item} (iii) \textit{and the upper bound in} (ii). By Theorem \ref{thm:accuracy} with
$N_*=\lceil\delta^{-1/\rho}\rceil$, $\gamma_1=1$, $\gamma_2=2$ the error of a
fixed method of the class $\mathcal S^\theta$ with $\theta>\beta$ does not exceed
$C\delta^{e_{\mathrm{tr}}}$ for $s,p\ne4/3$ and $C\delta^{e_{\mathrm{tr}}}(\log(\mathrm e/\delta))^{1/4}$
for $s=4/3$ or $p=4/3$; this is (iii). A finite sum defines an operator on the whole
data space, so this method belongs to $\Psi$. The Fourier method $\nu_k^N=1$
satisfies \eqref{qual} with $C(\nu)=1$ and $\theta=\beta+1$, so its constant depends
only on $s,p,\mu$; these are the upper bounds in (ii), including at the critical points. $\Box$

\vskip 2mm

\begin{remark}\label{rem:wellposed}\rm
Following \cite[\S\,2.8]{LuP}, we call a problem well-posed if its solution can be approximated with
accuracy of order $O(\delta)$. Let $\mu>\max\{0,\bar a(s)\}$, as in Theorem~\ref{thm:diag}.
Theorem~\ref{thm:diag}\,(ii), (iii) implies that for $1\le p<4/3$ and $s\ne4/3$ the problem of
summation of the Fourier--Hermite series from perturbed coefficients in the metric of ${\cal L}_\infty(\mathbb R)$ is well-posed: here $e_{\mathrm{tr}}=1$,
so that
$$
c\,\delta\le\mathcal E_\delta(W_s^\mu,\Psi,C(\mathbb R),\ell_p)\le C\,\delta ,
$$
and the order $\delta$ is attained by every method of the class $\mathcal S^\theta$ with $\theta>\beta$ at the
truncation level $N_*$. The reason is that for $p<4/3$ the data error is not amplified:
the last part of estimate \eqref{eq:sigma} of Lemma~\ref{lem:noise} does not depend on $N$. For $p>4/3$
we have $e_{\mathrm{tr}}<1$, and the order $\delta$ is attained by no finite diagonal
method (Theorem~\ref{thm:diag}\,(i)), and for $s\ge4/3$ by no recovery
operator either (Theorem~\ref{thm:diag}\,(ii)), so that the problem is ill-posed. For $s<4/3<p$
this conclusion is established only for finite diagonal methods. For $p=4/3$
or $s=4/3$ a logarithmic factor remains between the bounds of Theorem~\ref{thm:diag}. For
comparison: in the problem of summation of Fourier--Laguerre series in the weighted uniform metric
\cite{SolodkyVolynets2026} well-posedness holds only for $p=1$.
\end{remark}

\section{Numerical examples}\label{sec:numerics}

Consider $s=p=2$ and $\mu=4$. Then $\bar a(s)=a_+(p)=1/4$, $\beta=15/4$,
$\rho=4$ and $e_{\mathrm{tr}}=15/16$. By Theorem~\ref{thm:diag}\,(ii), (iii)
diagonal methods are in this case order-optimal among all recovery operators.
We compare the Fourier method and the Zygmund method with the factors
$$
\nu_k^N=1-(k/N)^4,\qquad 0\le k\le N .
$$
Both methods satisfy \eqref{qual} with $\theta=4>\beta$;
the level is chosen by the theory, without fitting to the exact function:
$N_\delta=\lceil\delta^{-1/4}\rceil$.
The examples show the behaviour of the methods on two fixed functions with one
deterministic noise vector and do not measure the largest error on the class.

\subsection{Two functions on the whole line}

The first function, before normalization, has two maxima at $x=\pm\sqrt{17/2}$
and a minimum at $x=0$:
$$
u_1(x)=|x|^{17/2}e^{-x^2/2} .
$$
The second example is based on the problem of recovering the function $\sin x+2$ from inexact
Hermite coefficients in \cite[Table~1, function~$f_3$]{AmatoDeFeis}.
There the space with the weight $e^{-x^2}$ and orthonormal polynomials are used;
multiplication by $e^{-x^2/2}$ turns its coefficients into the coefficients of the Hermite functions.
Therefore in our normalization we take
$$
u_2(x)=e^{-x^2/2}(2+\sin x) .
$$

The coefficients of the first function are
$$
a_{1,2m}=\frac{(-1)^m2^m\Gamma(19/4)(-17/4)_m}
                  {\pi^{1/4}\sqrt{(2m)!}},\qquad a_{1,2m+1}=0,
$$
where $(a)_m=\prod_{j=0}^{m-1}(a+j)$, $(a)_0=1$.
The formula follows from the identity $H_{2m}(x)=(-1)^m2^{2m}m!L_m^{(-1/2)}(x^2)$
and integration of the finite sum for the Laguerre polynomial.
Since $|a_{1,2m}|\asymp m^{-5}$, the coefficient norm is finite exactly for
$\mu<9/2$ and diverges for $\mu\ge9/2$. Thus the first example has a
finite smoothness bound and belongs to our class for $\mu=4$.

For the second example
$$
a_{2,k}=\frac{\pi^{1/4}}{\sqrt{k!}}\left[
 2\,\mathbf1_{\{k=0\}}+e^{-1/4}\operatorname{Im}\left(\frac{i}{\sqrt2}\right)^k\right].
$$
Here $\mathbf1_{\{k=0\}}$ equals $1$ for $k=0$ and $0$ for $k\ge1$: the term
$2e^{-x^2/2}=2\pi^{1/4}h_0(x)$ contributes only to the coefficient $a_{2,0}$.
For the term $e^{-x^2/2}\sin x=\operatorname{Im}e^{-x^2/2+ix}$ we take $d=i$ below.
For the generating function $\sum_{k\ge0}H_k(x)t^k/k!=e^{2xt-t^2}$ we have
$\int_{\mathbb R}e^{2xt-t^2}e^{-x^2+dx}\,dx=\sqrt\pi e^{d^2/4+dt}$. Expanding both sides in
powers of $t$ and equating the coefficients of equal powers, we obtain
$\int_{\mathbb R}H_k(x)e^{-x^2+dx}\,dx=\sqrt\pi\,e^{d^2/4}d^k$; after division by
$(2^{k}k!\sqrt{\pi})^{1/2}$ the coefficient of the function $e^{-x^2/2+dx}$ equals
$$
\frac{\pi^{1/4}e^{d^2/4}}{\sqrt{k!}}\left(\frac d{\sqrt2}\right)^k .
$$
Since $|a_{2,k}|\le3\pi^{1/4}/\sqrt{k!}$, the coefficients of the second function decay faster than
any power of $k$, and its norm is finite for every $\mu>0$.
In this example, too, we take $\mu=4$.
Set $f_j=u_j/R_j$, where $R_j=\|u_j\|_{2,4}$, $j=1,2$, so that $\|f_j\|_{2,4}=1$;
$R_1\approx221549.52$ and $R_2\approx16.5466$.

\subsection{Noise, errors and figures}

One vector $v=(v_0,\ldots,v_{100})$ is chosen in advance.
Its components are pseudorandom values of the standard normal distribution;
they are produced by the generator \texttt{default\_rng} of the NumPy library with the seed $1892026$.
For each level
$\delta$ we normalize its used initial part:
$$
\xi_k=\frac{\delta v_k}{(\sum_{j=0}^{N_\delta}v_j^2)^{1/2}},\quad 0\le k\le N_\delta,
\qquad \xi_k=0,\quad k>N_\delta .
$$
Thus the noise norm on the used coefficients equals $\delta$ exactly;
both methods receive the same data.

Tables~\ref{tab:power_bumps}--\ref{tab:amato_sin} contain the numerical values of the uniform error
$\|f_j-S_{N_\delta}^\nu(\mathbf f_j+\boldsymbol\xi)\|_\infty$.
The maximum was estimated on $[-25,25]$ with step $1/256$ and compared with step $1/128$.
The gaps between the nodes were controlled by an estimate of the second derivative of the error,
and outside this interval by the decay of the functions and of the finite sums.

\begin{table}[ht]
\centering
\caption{Errors of recovery of the function $f_1$}\label{tab:power_bumps}
\begin{tabular}{p{2cm}p{1cm}p{4cm}p{4cm}}
\hline $\delta$ & $N_\delta$ & Fourier & Zygmund, $\theta=4$ \\ \hline
$10^{-4}$ & 10 & $7.127\cdot10^{-5}$ & $5.528\cdot10^{-5}$ \\
$10^{-5}$ & 18 & $7.002\cdot10^{-6}$ & $6.040\cdot10^{-6}$ \\
$10^{-6}$ & 32 & $7.493\cdot10^{-7}$ & $7.180\cdot10^{-7}$ \\
$10^{-7}$ & 57 & $6.990\cdot10^{-8}$ & $7.099\cdot10^{-8}$ \\
$10^{-8}$ & 100 & $6.193\cdot10^{-9}$ & $7.394\cdot10^{-9}$ \\
\hline\end{tabular}
\end{table}

\begin{table}[ht]
\centering
\caption{Errors of recovery of the function $f_2$}\label{tab:amato_sin}
\begin{tabular}{p{2cm}p{1cm}p{4cm}p{4cm}}
\hline $\delta$ & $N_\delta$ & Fourier & Zygmund, $\theta=4$ \\ \hline
$10^{-4}$ & 10 & $7.084\cdot10^{-5}$ & $9.515\cdot10^{-5}$ \\
$10^{-5}$ & 18 & $7.013\cdot10^{-6}$ & $1.019\cdot10^{-5}$ \\
$10^{-6}$ & 32 & $7.500\cdot10^{-7}$ & $1.025\cdot10^{-6}$ \\
$10^{-7}$ & 57 & $6.996\cdot10^{-8}$ & $1.070\cdot10^{-7}$ \\
$10^{-8}$ & 100 & $6.193\cdot10^{-9}$ & $9.746\cdot10^{-9}$ \\
\hline\end{tabular}
\end{table}

\IfFileExists{power_bumps.pdf}{%
\IfFileExists{amato_sin.pdf}{%
Figures~\ref{fig:power_bumps}--\ref{fig:amato_sin} show the neighbouring levels
$10^{-4}$ and $10^{-5}$. At the coarser noise the deviation of both methods does not exceed
$15\%$ of the maximum modulus of the exact function, and at the next level the error decreases
for both methods by more than a factor of $1.5$. The first neighbouring pair
satisfying these two conditions is chosen. The upper panels show the functions,
the lower ones the differences with the exact function; the axis limits within each pair are the same.
For the smooth second example the improvement is visible in the lower panels.

\begin{figure}[ht]
\centering
\includegraphics[width=\textwidth]{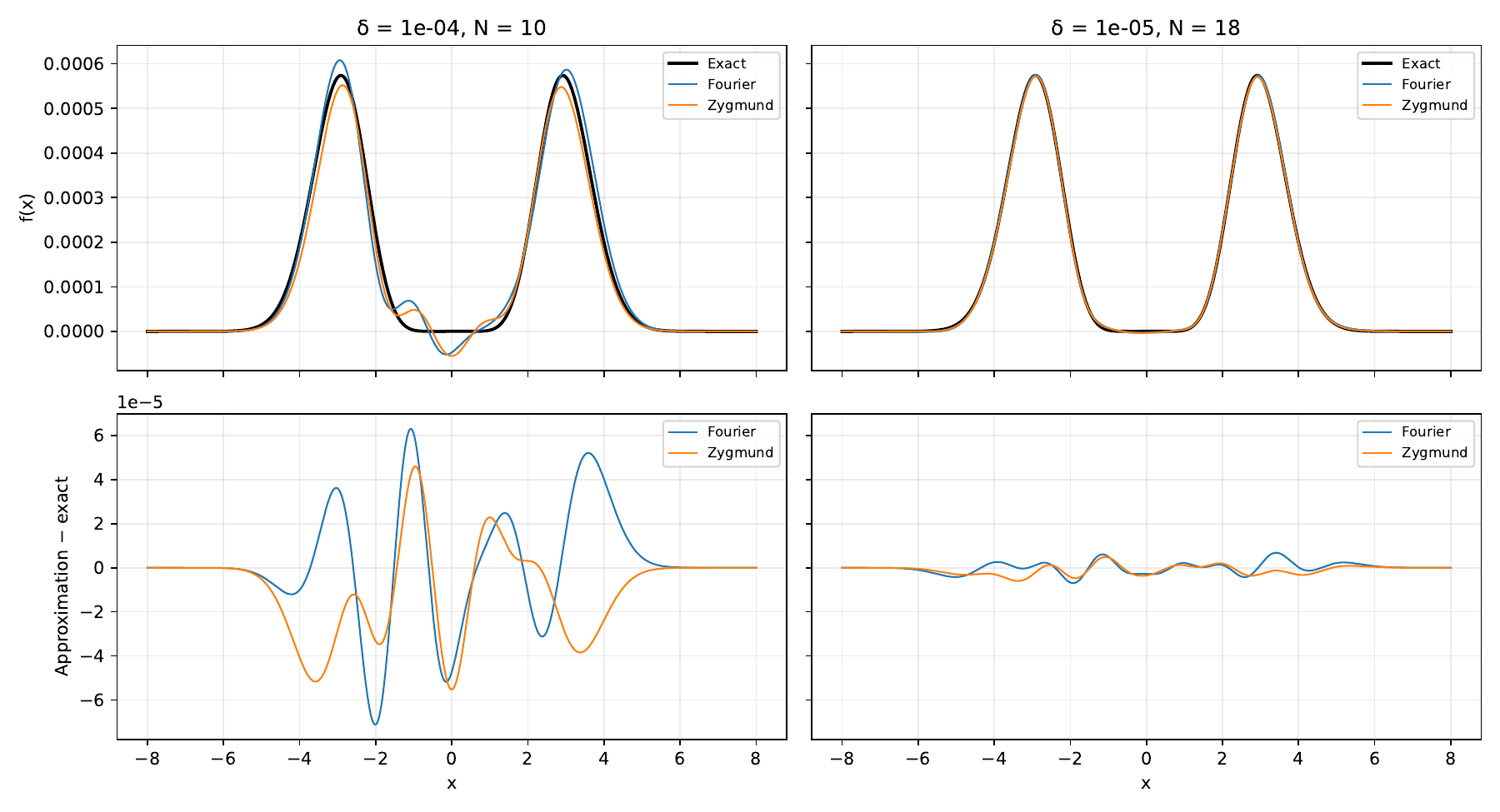}
\caption{Recovery of $f_1$: two maxima and a minimum between them.
Left $\delta=10^{-4}$, right $10^{-5}$.}\label{fig:power_bumps}
\end{figure}

\begin{figure}[ht]
\centering
\includegraphics[width=\textwidth]{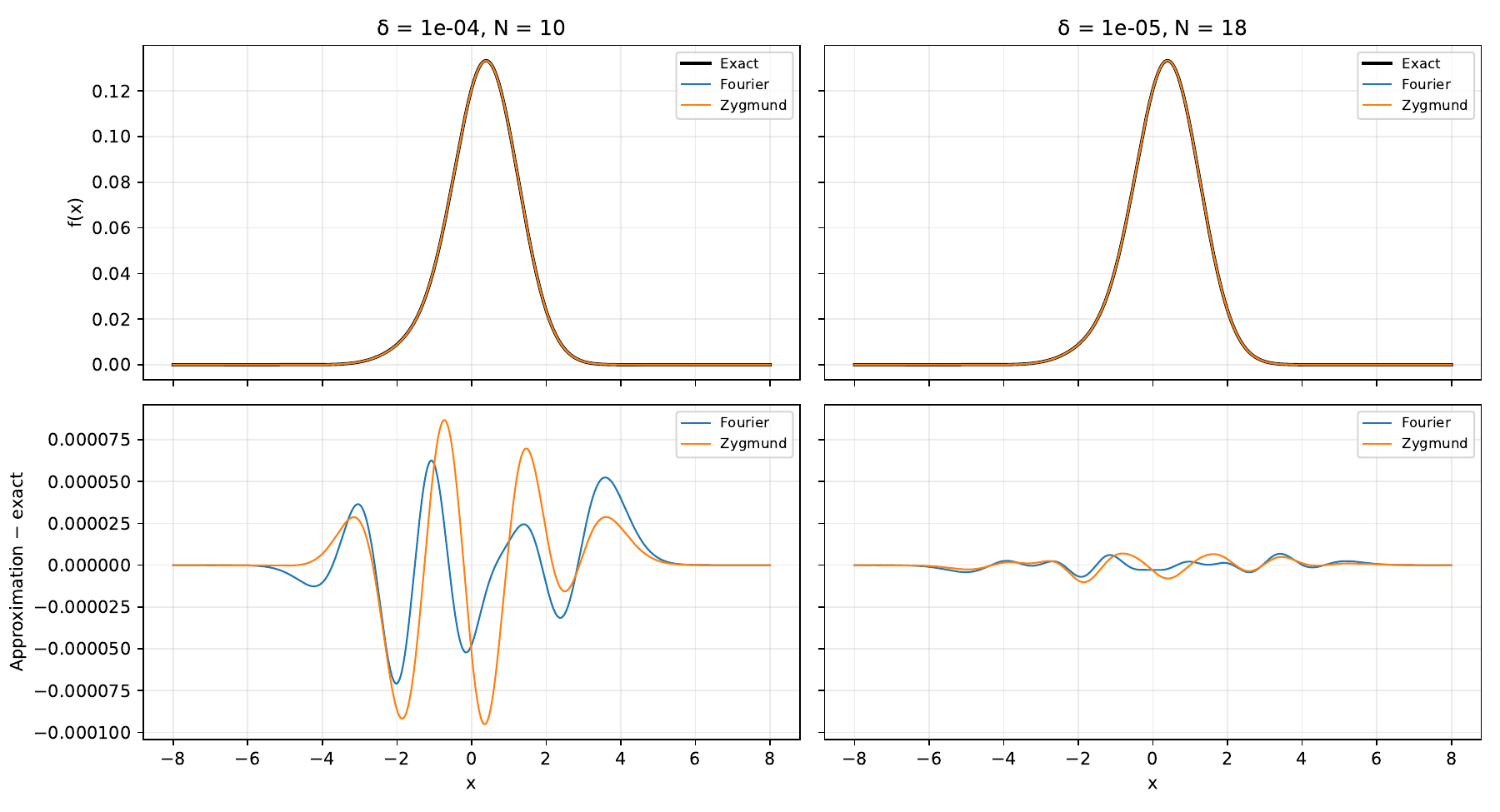}
\caption{Recovery of $f_2$, adapted from the example of Amato--De Feis:
left $\delta=10^{-4}$, right $10^{-5}$.}\label{fig:amato_sin}
\end{figure}

}{\typeout{Figures omitted: amato_sin.pdf not supplied.}}%
}{\typeout{Figures omitted: power_bumps.pdf not supplied.}}

\FloatBarrier

\end{document}